\documentclass[journal]{IEEEtran}
\usepackage{CJK}
\usepackage{amssymb}
\usepackage{epstopdf}
\usepackage{graphicx}
\usepackage{multirow}
\usepackage{diagbox}
\usepackage{algorithmicx,algorithm}
\usepackage{framed}
\usepackage{booktabs}
\usepackage{threeparttable}
\usepackage{comment}
\usepackage{amsmath}
\usepackage{amsthm}
\usepackage{xcolor}

\graphicspath{ {./images/} }

\newtheorem{thm}{Theorem}
\newtheorem{lem}{Lemma}
\newtheorem{cor}{Corollary}
\newtheorem{pro}{Proposition}

\newtheorem{rmk}{Remark}

\newtheorem{ass}{Assumption}

\newcommand {\emptycomment}[1]{}

\newcommand{\be }{\begin{equation}}
\newcommand{\ee }{\end{equation}}

\newcommand{\noi}{\noindent}

\newcommand{\ttt}{\theta}
\newcommand{\wh}{\widetilde}
\newcommand{\p}{\partial}
\newcommand{\f}{\frac}

\newcommand{\aaa}{\alpha}

\newcommand{\nn}{\langle}
\newcommand{\mm}{\rangle}
\newcommand{\nono}{\nonumber}
\newcommand{\huaX}{\mathcal{X}}

\newcommand{\mbb}{\mathbb}

\newcommand{\kkk}{\kappa}

\def\bea{\begin{eqnarray}}
\def\eea{\end{eqnarray}}
\def\be{\begin{equation}}
\def\ee{\end{equation}}
\def\blm{\begin{lem}}
\def\elm{\end{lem}}
\def\btm{\begin{theorem}}
\def\etm{\end{theorem}}

\def\p{\mathcal{P}}

\def\ff{\Phi}

\newcommand{\huaF}{\mathcal{F}}

\newcommand{\huaP}{\mathcal{P}}

\newcommand{\huaI}{\mathcal{I}}

\newcommand{\huaO}{\mathcal{O}}
\newcommand{\huaT}{\mathcal{T}}

\def\bea{\begin{eqnarray}}
\def\eea{\end{eqnarray}}
\def\be{\begin{equation}}
\def\ee{\end{equation}}
\def\blm{\begin{lem}}
\def\elm{\end{lem}}

\def\bea{\begin{eqnarray}}
	\def\eea{\end{eqnarray}}
\def\be{\begin{equation}}
	\def\ee{\end{equation}}
\def\blm{\begin{lem}}
	\def\elm{\end{lem}}
\def\bes{\begin{eqnarray*}}
	\def\ees{\end{eqnarray*}}
\def\beal{\begin{aligned}}
	\def\eeal{\end{aligned}}
\def\ppp{\Phi}
\def\wh{\widehat}
\def\ff{\Phi}

\def\nb{\nabla}

\def\ww{\widetilde}
\def\sss{\sigma}

\def\nnn{\nabla}
\newtheorem{theorem}{Theorem}

\begin{document}
\begin{CJK*}{GBK}{song}

\title{Distributed Stochastic Optimization under Heavy-Tailed Noise: A Federated Mirror Descent Approach with High Probability Convergence
}
\author{Zhan Yu, Lan Liao, Deming Yuan, \IEEEmembership{Senior Member, IEEE}, Daniel W. C. Ho, \IEEEmembership{Life Fellow, IEEE}, Ding-Xuan Zhou\thanks{Zhan Yu is with the Department of Mathematics, Hong Kong Baptist University.  Lan Liao and Deming Yuan are with the School of Automation, Nanjing University of Science and Technology. Daniel W. C. Ho is with the Department of Mathematics, City University of Hong Kong. Ding-Xuan Zhou is with the School of Mathematics and Statistics, The University of Sydney.
}}
\maketitle


\begin{abstract}
We study the distributed stochastic optimization (DSO) problem  under a heavy-tailed noise condition by utilizing a multi-agent system. Despite the extensive research on DSO algorithms used to solve DSO problems under light-tailed noise conditions (such as Gaussian noise), there is a significant lack of study of DSO algorithms in the context of heavy-tailed random noise. 
Classical DSO approaches in a heavy-tailed setting may present poor convergence behaviors. Therefore, developing DSO methods in the context of heavy-tailed noises is of  importance. This work follows this path and we consider the setting that the gradient noises associated with each agent can be heavy-tailed, potentially having unbounded variance.  We propose a clipped  federated stochastic mirror descent algorithm to solve the DSO problem. We rigorously present a convergence theory and show that, under appropriate rules on the stepsize and the clipping parameter associated with the local noisy gradient influenced by the heavy-tailed noise, the algorithm is able to achieve satisfactory high probability convergence.

\end{abstract}
\begin{IEEEkeywords}
Distributed optimization, federated learning, heavy-tailed noise, stochastic optimization, stochastic gradient, mirror descent
\end{IEEEkeywords}

\section{Introduction}
In the past decade, distributed optimization has achieved a remarkable success across various fields, including systems science, data science, and machine learning \cite{s1}-\cite{sc2024}.   This success is largely due to the efficiency of distributed optimization in managing large-scale data, information, and decision variables within application domains such as the Internet of Things \cite{yyw2019},       data-driven machine learning \cite{kbnrs2020}, smart grid \cite{cns2014}, power system \cite{yhl2016}. A core objective of distributed optimization is to enable local agents or agents to collaborate through suitable cooperation patterns and algorithmic mechanisms to minimize a global objective function in an appropriate decision space. This global function can often be decomposed into the sum of local objective functions associated with the individual agents. To address this problem, many effective distributed algorithms have been developed over the past decade. Various fundamental optimization algorithmic frameworks, such as (sub-)gradient descent \cite{s1}, \cite{rnv2010}, mirror descent \cite{yhhx2020}, (sub-)gradient push \cite{no2015}, primal-dual methods \cite{cns2014}, dual-averaging \cite{daw2012}, and gradient tracking \cite{pn2021}, have been employed to construct efficient distributed optimization and learning methods.

In the existing literature on DSO algorithms, distributed stochastic gradient (DSG) methods play a crucial role. We observe that, a majority of existing works on stochastic gradient based methods in DSO are based on the assumption of light-tailed stochastic gradient noise. This can be traced back to the earlier foundational work \cite{rnv2010} and encompasses the main-stream developments of DSO in recent years, including e.g., \cite{daw2012}, \cite{yhhx2020}, \cite{llww2018}, \cite{xzhyx2022}, \cite{lwzw2023}. An obvious feature of the light-tailed noises is that the corresponding distribution satisfies the bounded variance assumption. 
 However, the random environments in real-world machine learning and systems science applications are often more complex. Accordingly,  examples from statistical learning and deep learning often exhibit heavy-tailed random noises, potentially possessing unbounded variance (see e.g., \cite{ssg2019}-\cite{zhang2020}). The emergence of heavy-tailed randomness causes classical DSO methods to encounter bottlenecks in their corresponding handling. In this context, efforts to address DSO under heavy-tailed noise has only recently begun to emerge \cite{qlxc2025}-\cite{sc2024}.
  The existing theoretical foundations based on light-tailed noise models such as the sub-Gaussian noise are quite idealized and no longer sufficient to support the rapidly evolving practical demands of machine learning development under heavy-tailed noise.  Therefore, it is of significance to further develop effective DSO algorithms with stochastic gradients induced by heavy-tailed noise. This work will proceed along this path. 
  
  In this work, we aim to develop DSG methods for solving DSO problem when the stochastic gradient is influenced by heavy-tailed noise. Our proposed method is related to  federated optimization/learning, which has recently emerged as a very popular framework in machine learning for addressing finite-sum minimization in DSO  \cite{lszsts2020}-\cite{lsts2020}. Our  model  involves a group of local agents (clients) and a coordinate server (master). In our algorithm, the local agents collaboratively train a shared model. Instead of communicating directly with one another, the local agents periodically send their local information to the coordinate server, the coordinate server is able to aggregate these local updates in some appropriate manners and sends the averaged information back to local agents.  Such an update procedure gives this algorithm a strong privacy protection mechanism.  The idea of federated optimization/learning has also been witnessed in the well-developed divide-and-conquer based distributed learning algorithms for handling large-scale data in statistical learning  \cite{yfsz2024}, \cite{llwzz2025}.
  
   For handling the stochastic gradient influenced by the heavy-tailed random noise, we utilize a  stochastic gradient clipping strategy (\cite{qlxc2025}-\cite{sc2024}, \cite{nnen2023}, \cite{gsdhgdgr2024}) to mitigate the impact of the heavy-tailed random noise. The idea of related truncating techniques have been considered for addressing heavy-tailed noise or noise under weak moment conditions in statistical learning theory  \cite{fw2020}, \cite{hfw2022}, and it has only received attention in the field of DSO in the last two years (\cite{qlxc2025}-\cite{sc2024}) and has potential vitality in DSO. For the algorithm structure in each agent, we employ the mirror descent one induced by Bregman divergence to  perform the updates required to complete each iteration of our algorithm. Mirror descent, introduced by Nemirovski and Yudin in the 1980s, emerged as a powerful optimization technique in convex analysis due to its flexibility with geometric structures. Compared with other methods, mirror descent can better adapt to various geometric structures of the underlying space by choosing appropriate mirror maps. This flexibility allows it to perform well in diverse fields, including machine learning, statistics, and signal processing. Based on this fact, the theory presented in this paper naturally encompasses many important federated optimization algorithmic theories for DSO under heavy-tailed randomness, including federated stochastic gradient descent and federated stochastic entropic descent. To our knowledge, despite significant progress in federated stochastic optimization algorithms in the past five years, research on federated stochastic mirror descent under heavy-tailed gradient noise has not yet been effectively pursued. This paper will provide new insights into this branch of research.

We summarize the contributions of this work as follows:

\begin{itemize}
	\item A clipped federated stochastic mirror descent algorithm (Clipped FedSMD) is proposed for solving the DSO problem under heavy-tailed random noise which may have unbounded variance.  The high probability convergence is rigorously established for the Clipped FedSMD. We  rigorously develop a set of inductive proof techniques for the high probability convergence of the federated optimization approach to DSO under heavy-tailed randomness. Explicit selection rules for stepsizes and clipping parameters are provided. Due to the influence of the consensus bound, the selections are essentially different from the non-distributed methods for handling heavy-tailed noises. Moreover, this selection does not require a known time-horizon that the existing works often required. We demonstrate  that the algorithm can achieve a high probability convergence rate of 
	 $\huaO\left(T^{\f{1-p}{2p}}\log^\gamma T\right)$ for solving the DSO problem where $1<p\leq2$ is the tail parameter used to measure the heavy-tailed randomness and $\gamma>1$ is a freely chosen stepsize parameter.

	\item 
Due to the introduction of a mirror map and its associated Bregman divergence, the Clipped FedSMD has the potential to achieve efficient updating advantages when clear geometric feature of the decision space is present in the DSO problem, provided that an appropriate selection of the mirror map is made. By introducing an appropriate monotonic sequence relying on the Bregman divergence in the proof, our analysis techniques do not require the compactness condition on the decision space that is commonly used in the majority of
 existing works on distributed mirror descent (DMD). The theory is applicable to constrained DSO in both bounded and unbounded decision domains, as well as to unconstrained DSO.

\end{itemize}

\noi \textbf{Notation:} Denote the standard $n$-dimension Euclidean space by ($\mathbb R^n$, $\|\cdot\|$). Denote the set of all positive integers by $\mbb N$. We use $\nn\cdot,\cdot\mm$ to denote the standard Euclidean inner product. 
For a matrix $M\in\mathbb R^{m\times m}$, we denote the element in $i$th row and $j$th column by $[M]_{ij}$. For an $n$-dimension Euclidean vector $x$, denote its $i$-th component by $[x]_i$, $i=1,2,...,n$. We use $\mbb P(S)$ to denote the probability of a measurable set $S$. 

\section{Problem setting}
In this work, we focus on a system consisting of $m$ local agents (clients) and a coordinate server. Around this system, we study the following multi-agent optimization problem
\begin{eqnarray} 
	\min_{x\in \huaX} f(x)=\sum_{i=1}^mf_i(x). \label{problem} 
\end{eqnarray}
In \eqref{problem}, $x\in \huaX\subset\mathbb R^n$ is a global decision vector, $\huaX$ is a closed convex decision domain. The function $f_i: \huaX\to \mathbb R$ is the convex objective function known only to the $i$th agent (client).    We suppose that there is at least an optimal point $x^*\in\huaX$ such that $f(x)\geq f(x^*)$ for $x\in\huaX$.

 At each iteration step, $t\geq1$, each agent $i\in[m]$ only has access to a noisy gradient of $f_i$, denoted by $\wh \nabla f_i:\mbb R^n\rightarrow\mbb R^n$. For this stochastic gradient, we make the following basic assumption.
\begin{ass}\label{stochastic_gradient}
For any $i\in[m]$, $\wh \nabla f_i(x)$ is an unbiased gradient estimator in the sense that $\mbb E[\wh \nabla f_i(x)|x]=\nabla f_i(x)$. Moreover, $\wh \nabla f_i(x)$ satisfies the following bounded $p$-moment condition:
\bes
\mbb E\left[\left\|\wh \nabla f_i(x)-\nabla f_i(x)\right\|^p\Big|x\right]\leq\sigma^p
\ees
for some $\sigma>0$ and $p\in(1,2]$.
\end{ass} 

For the local objective functions $f_i$, $i\in[m]$, we consider the following widely-adopted $L$-smoothness condition.
\begin{ass}\label{Lsmooth_condition}
Local objective functions	$f_i$, $i\in[m]$ satisfy the $L$-smoothness condition,  namely, for any $x,y\in\mbb \huaX$,
	\bes
	\|\nabla f_i(x)-\nabla f_i(y)\|\leq L\|x-y\|.
	\ees
	\end{ass} 
In this work, for dealing with the heavy-tailed randomness, we will introduce the gradient clipping operator for handling the stochastic gradients $\wh \nabla f_i$, $i\in[m]$. For each agent $i\in[m]$, let $x_{i,t}$ be its state variable at time instant $t$. We use $\wh \nabla f_i(x_{i,t})$ to denote the stochastic gradient via querying the stochastic gradient oracle. Then we represent the clipped gradient estimator $\ww \nabla_{\lambda_t} f_i(x_{i,t})$ associated with the clipping operator $\ww\nabla_{\lambda_t}$ with the clipping parameter $\lambda_t$ at iteration $t$ by
\bes
\ww \nabla_{\lambda_t} f_i(x_{i,t})=\min\left\{1,\f{\lambda_t}{\|\wh \nabla f_i(x_{i,t})\|}\right\}\wh \nabla f_i(x_{i,t}).
\ees

In this paper, the Bregman divergence $D_\Phi(\cdot\|\cdot)$ associated with a differentiable function $\Phi$ is crucial for formulating the main algorithm (see e.g.,  \cite{llww2018}, \cite{yhy2022},  \cite{hljl2025}).
For a differentiable  function $\ff:\huaX\rightarrow\mbb R$, the Bregman divergence $D_\Phi(x\|y)$ between $x\in\huaX$ and $y\in\huaX$ induced by $\Phi$ is defined as $$D_\Phi(x\|y)=\Phi(x)-\Phi(y)-\nn\nabla\Phi(y), x-y\mm.$$ 
We always call such function $\ff$ a mirror map (distance generating function). Throughout this paper, without loss of generality, we always assume
$\Phi:\huaX\to\mathbb R$ be a  $1$-strongly convex differentiable function, namely, $\Phi(x)-\Phi(y)\geq\nn\nabla\Phi(y), x-y\mm+\f{1}{2}\|x-y\|^2.$ Here, we consider $1$-strong convexity merely for the sake of brevity in the writing, and for general parameter settings, we can achieve corresponding analysis through a trivial scaling process. Such consideration has been adopted in many references on mirror descent e.g. \cite{nnen2023}, \cite{dajj2012}. According to the structure of $D_\Phi(\cdot\|\cdot)$, a direct consequence of the $1$-strong convexity of $\ff$ is the useful relation between Bregman divergence and the classical Euclidean distance:  
$$D_{\Phi}(x\| y)\geq \f{1}{2}\|x-y\|^2.$$ Another useful result about Bregman divergence following directly from the above notion is listed in the following lemma. 
\begin{lem}\label{bregman_3point}
	The Bregman divergence $D_\ppp(\cdot\|\cdot)$ satisfies the three-point identity $$\nn\nabla\Phi(x)-\nabla\Phi(y),y-z\mm=D_{\Phi}(z\|x)-D_{\Phi}(z\|y)-D_{\Phi}(y\|x)$$
	for all $x, y, z\in \huaX$.
\end{lem}

The following separate convexity assumption is a standard and widely adopted one in the literature of DMD methods, as seen in works such as  \cite{yhhx2020}, \cite{md1}, \cite{lj2021}. It is particularly useful for developing the related convergence theory and many common mirror maps naturally satisfy this assumption.
\begin{ass}\label{separable_convexity_ass}
	The Bregman divergence $D_{\Phi}(\cdot\|\cdot)$  satisfies the separate convexity with respect to the second variable, namely, for any $x\in\huaX$, any $y_j\in\huaX$, $j\in[m]$ and any $a_j\geq 0$, $j\in[m]$ which satisfies $\sum_{j=1}^ma_j=1$,  it holds that $$D_{\Phi}(x\| \sum_{j=1}^ma_jy_j)\leq\sum_{j=1}^ma_jD_{\Phi}(x\|y_j).$$
\end{ass}


\section{The main algorithm and convergence analysis}
In this section, we first introduce the update mechanism of the algorithm, followed by the the  convergence results and their related overall convergence analysis and proofs.

We develop a so-called clipped federated stochastic mirror descent (Clipped FedSMD) method to solve the multi-agent optimization problem \eqref{problem}. The updating rule of the Clipped FedSMD method is described in Algorithm 1. In our method, we require communications between the central server and the local agents  periodically. We set the communication period to be a constant $\p\in\mbb N$. Then we denote the set of   communication instants  by
\be\mathcal{I}=\left\{t_{1}, \ldots, t_{\huaT}\right\} \label{defhuat}
\ee
with  $$t_i=1+i\p, \  \ i=1,2,...,\huaT$$ and $T=t_{\huaT}$. Thus, over the time horizon $T$, there are precisely $\huaT$ communication rounds. Between two successive communication rounds, each agent independently performs a type of clipped SMD algorithm. When the time instant falls within the communication instant set $\huaI$,  the local agents  upload their local state variables to the coordinate server. The coordinate server collects these local information of state variables and performs an averaging operation to update a new average state variable. Subsequently, the coordinate server sends the averaged variable to these local agents. The local agents receive the average variable and perform further updates to complete an iteration step.

\begin{algorithm}[htbp]
	\caption{Clipped Federated Stochastic Mirror Descent}
	\begin{algorithmic}\label{main_algorithm}
		\State{\textbf{Initialization}: Initialize points $x_{i,1}\in \huaX$, $ i\in [m]$, step sizes $\{\aaa_t\}$, clipping parameters $\{\lambda_t\}$,  fix the communication period $\p$, input the total iteration $T=\p\huaT+1$ with total  communication rounds $\huaT$.}
		
		\State{\textbf{Updating rule}: For $k=1,...,T$, for agents $i=1,2,...,m$:
			
			\noi (1) Compute the local clipped stochastic gradient 
			$$\ww \nabla_{\lambda_t} f_i(x_{i,t})=\min\left\{1,\f{\lambda_t}{\|\wh \nabla f_i(x_{i,t})\|}\right\}\wh \nabla f_i(x_{i,t}).$$
			
			\noi (2) Compute auxiliary local variable $y_{i, t+1}$ as
			$$
			y_{i,t+1}=\arg\min_{x\in \huaX}\Big\{\left\nn \ww \nabla_{\lambda_t} f_i(x_{i,t}),x\right\mm+\f{1}{\aaa_t}D_{\Phi}(x\|x_{i,t})\Big\}.
			$$
			
			\noi (3)  If $t+1 \in \mathcal{I}$ then upload $y_{i, t+1}$ to the coordinate server, receive $\f{1}{m} \sum_{i =1}^m y_{i, t+1}$ from the server, and update the state variable  as
			$$
			x_{i, t+1}=\bar{y}_{t+1} := \frac{1}{m} \sum_{i=1}^{m} y_{i, t+1}.
			$$
			Else, update the local variable by setting $$x_{i,t+1}=y_{i,t+1}.$$
		}
		\end{algorithmic}
\end{algorithm}

It is worth mentioning that during each update step of our algorithm, there is no direct information exchange among the local agents, as is witnessed in some distributed optimization algorithms. Instead, at the communication instant time points, the agents communicate indirectly through the coordinate server as a mediator. This federated learning update mechanism provides our algorithm with a strong privacy protection advantage when addressing DSO problems in heavy-tailed noise environments, establishing a structural foundation for the information security of the agents.

Now we conduct the convergence analysis of the algorithm. For handling the stochasticity, let us define $\huaF_t=\sigma(\bigcup_{i=1}^m\{\wh\nabla f_i(x_{i,1}), \wh\nabla f_i(x_{i,2}), ..., \wh\nabla f_i(x_{i,t})\})$,  the sigma-algebra generated by the entire history of the randomness till step $t$ by $\huaF_t$ for $t\geq1$,  
and $\huaF_0=\{\emptyset,\huaX\}$. Thus, we can see that $\huaF_t$ naturally forms a filtration. Equipped with the above notions, for each agent $i\in[m]$ and iteration $t$, we introduce the following three notions for further handling the stochasticity arising from the stochastic gradient, namely, the stochastic error $\ttt_{i,t}$, the deviation $\ttt_{i,t}^u$ and the bias of the clipped stochastic gradient estimator $\ttt_{i,t}^b$:
\bea
&&\ttt_{i,t}=\ww\nabla_{\lambda_t} f_i(x_{i,t})-\nabla f_i(x_{i,t}),\\
&&\ttt_{i,t}^u=\ww \nnn_{\lambda_t} f_i(x_{i,t})-\mbb E[\ww \nnn_{\lambda_t} f_i(x_{i,t})|\huaF_{t-1}],\\
&&\ttt_{i,t}^b=\mbb E[\ww\nnn_{\lambda_t} f_i(x_{i,t})|\huaF_{t-1}]-\nnn f_i(x_{i,t}).
\eea
The above three quantities are useful for handling the clipped gradients $\ww \nabla_{\lambda_t} f_i(x_{i,t})$, $i\in[m]$. It is easy to see from the above notions, there holds the basic relation $\ttt_{i,t}=\ttt_{i,t}^u+\ttt_{i,t}^b$. The next lemma provides basic upper bound estimates related to $\ttt_{i,t}$, $\ttt_{i,t}^u$ and $\ttt_{i,t}^b$ (see e.g., \cite{ylw2025}, \cite{nnen2023}, \cite{zhang2020}). 
\begin{lem}\label{bound_gradient_lemma}
	Under Assumptions \ref{stochastic_gradient} and \ref{Lsmooth_condition}, for $t\geq1$ and $i\in[m]$, it holds that
	\bes
	\|\ttt_{i,t}^u\|\leq2\lambda_t.
	\ees
	Moreover, if $\|\nnn f_i(x_{i,t})\|\leq\f{\lambda_t}{2}$, $i\in[m]$, then it holds that
	\bes
	&&\|\ttt_{i,t}^b\|\leq4\sigma^p\lambda_t^{1-p},\\
	&&\mbb E[\|\ttt_{i,t}^u\|^2|\huaF_{t-1}]\leq40\sigma^p\lambda_t^{2-p}.
	\ees
\end{lem}

In our subsequent analysis, we also require a useful inequality for handling randomness, which is listed in the following lemma.
 (see e.g., \cite{nnen2023}, \cite{bllrs2011}).
\begin{lem}\label{random_ine_lemma}
	Suppose $X$ is a random variable which satisfies $\mbb E[X]=0$ and $|X|\leq M$ almost surely. Then for $0\leq\lambda\leq\f{1}{M}$, it holds that
	\bes
	\mbb E[\exp(\lambda X)]\leq\exp(\f{3}{4}\lambda^2\mbb E[X^2]).
	\ees
\end{lem}

Before giving our first result, let us recall the definition of $\huaI$ that represents  the set of communication time instants in \eqref{defhuat}. we define a sequence $\tau(t)$ by
\be
\tau(t)=\left\{
\begin{aligned}
	\max\{t'\in\huaI:t'\leq t\}, &\quad \mbox{if} \ t>1+\p,\\
	1, \qquad \qquad &\quad \mbox{if} \ 1\leq t\leq1+\p.\\
\end{aligned}
\right. \label{deftau}
\ee
The following result provides an important consensus bound for the state deviation from the average state for the Clipped FedSMD  in the context of heavy-tailed stochastic gradient noises.

\begin{pro}\label{basic_consensus}
Under Assumption \ref{stochastic_gradient}, for any $i\in[m]$, it holds that
	\bes
&&\|x_{i,t}-\bar x_t\|\leq 2\sum_{s=\tau(t)}^{t-1}\aaa_s\lambda_s.
\ees
\end{pro}
\begin{proof}
	For $t\in\huaI$, it is trivial to see from the Clipped FedSMD algorithm that $x_{i,t}=\bar{x}_t$. For $t\notin\huaI$,  we have that
$$x_{i,t}=\bar{x}_{\tau(t)}+\sum_{s=\tau(t)}^{t-1}(x_{i,s+1}-x_{i,s}).$$
Taking averages of the above equation among agents, we have
\be
\bar{x}_{t}=\bar{x}_{\tau(t)}+\f{1}{m}\sum_{j=1}^m\sum_{s=\tau(t)}^{t-1}(x_{j,s+1}-x_{j,s}).\label{noexp}
\ee
The above two equations imply that
\bea
\begin{aligned}
\|x_{i,t}-\bar{x}_t\|\leq&\left(1-\f{1}{m}\right)\sum_{s=\tau(t)}^{t-1}\|x_{i,s+1}-x_{i,s}\|\\
&+\f{1}{m}\sum_{j\in[m]\setminus \{i\}}\sum_{s=\tau(t)}^{t-1}\|x_{j,s+1}-x_{j,s}\|.\label{consensus_bdd}
\end{aligned}
\eea
For $s=\tau(t), \ldots, t-1$, since $s+1\notin\huaI$, from the iterative updating rule of the main algorithm, we have
$$x_{i,s+1}=\arg\min_{x\in \huaX}\Big\{\left\nn \ww \nabla_{\lambda_t} f_i(x_{i,s}),x\right\mm+\f{1}{\aaa_s}D_{\Phi}(x\|x_{i,s})\Big\}.
$$
Applying the first-order optimality condition to the above equation implies that, for any $x\in\huaX$, it holds that
\bes
\left\nn\aaa_s \ww \nabla_{\lambda_t} f_i(x_{i,s})+\nabla\ff(x_{i,s+1})-\nabla\ff(x_{i,s}),x-x_{i,s+1}\right\mm\geq0.
\ees
By setting $x=x_{i,s}$ in the above inequality, we have
\bes
&&\nn\nb\ff(x_{i,s})-\nb\ff(x_{i,s+1}),x_{i,s}-x_{i,s+1}\mm\\
&&\leq\left\nn \aaa_s \ww \nabla_{\lambda_t} f_i(x_{i,s}),x_{i,s}-x_{i,s+1}\right\mm.
\ees
Applying the Schwarz inequality to the right hand side and $1$-strong convexity of $\ff$ to the left hand side of above inequality, we have
\bes
\|x_{i,s}-x_{i,s+1}\|^2\leq\aaa_s\|\ww \nabla_{\lambda_t} f_i(x_{i,s})\|      \|x_{i,s}-x_{i,s+1}\|.
\ees
Eliminating the same term $\|x_{i,s}-x_{i,s+1}\|$ on both sides of the above inequality, we obtain that 
\bes
\|x_{i,s}-x_{i,s+1}\|\leq\aaa_s\|\ww \nabla_{\lambda_t} f_i(x_{i,s})\|\leq\aaa_s\lambda_s.      
\ees
Here we have used the fact that
\bes
\beal
\|\ww \nabla_{\lambda_t} f_i(x_{i,s})\|=&\left\|\min\left\{1,\f{\lambda_s}{\|\wh \nabla f_i(x_{i,s})\|}\right\}\wh \nabla f_i(x_{i,s})\right\|\leq\lambda_s.
\eeal
\ees
Now we go back to \eqref{consensus_bdd} and substitute the above relation, we finally obtain that
	\bes
\begin{aligned}
	\mbb \|x_{i,t}-\bar{x}_t\|\leq&\left(1-\f{1}{m}\right)\sum_{s=\tau(t)}^{t-1}\aaa_s\lambda_s\\
	&+\f{1}{m}\sum_{j\in[m]\setminus\{i\}}\sum_{s=\tau(t)}^{t-1}\aaa_s\lambda_s\leq2\sum_{s=\tau(t)}^{t-1}\aaa_s\lambda_s.
\end{aligned}
\ees
We finish the proof.
\end{proof}

The above result is a crucial consensus bound related to the state variables of all the agents in terms of the stepsizes and clipping parameters. We need to point out that it is precisely due to the influence of this consensus bound in this distributed algorithm that the selection of the corresponding step size and clipping parameters differs fundamentally from that in non-distributed scenarios (e.g. \cite{nnen2023}) because we need to provide a strictly decaying strategy for the consensus bound. This is crucial to ensure the final convergence of the Clipped FedSMD as witnessed in the subsequent analysis.

The following estimate is fundamental for establishing the convergence theory of the Clipped FedSMD method.
\begin{pro}\label{inner_product_est1}
Under Assumptions \ref{stochastic_gradient} and \ref{separable_convexity_ass}, there holds
\bes
\nono&&\sum_{i=1}^m\left\nn\aaa_t\ww\nabla_{\lambda_t} f_i(x_{i,t}),x_{i,t}-x^*\right\mm\\
&&\leq \sum_{i=1}^m\Big[D_{\ff}(x^*\|x_{i,t})-D_{\ff}(x^*\|x_{i,t+1})\Big]\\
&&\quad+\f{\aaa_t^2}{2}\sum_{i=1}^m\left\|\ww\nabla_{\lambda_t} f_i(x_{i,t})\right\|^2.
\ees	
\end{pro}
\begin{proof}
	According to the first-order optimality, for any $x\in\huaX$, it holds that
	\be
	\nono\nn\nb\ff(y_{i,t+1})-\nb\ff(x_{i,t})+\aaa_t\ww\nnn_{\lambda_t} f_i(x_{i,t}),x-y_{i,t+1}\mm\geq0.
	\ee
	Setting $x=x^*$ in above inequality, and rearranging terms, we have
	\bea
	\nono&&\left\nn\aaa_t\ww\nnn_{\lambda_t} f_i(x_{i,t}),y_{i,t+1}-x^*\right\mm   \\
	\nono&&\leq\nn\nb\ff(y_{i,t+1})-\nb\ff(x_{i,t}),x^*-y_{i,t+1}\mm\\
	\nono&&= D_{\ff}(x^*\|x_{i,t})-D_{\ff}(x^*\|y_{i,t+1})-D_{\ff}(y_{i,t+1}\|x_{i,t})\\
	\nono&&\leq D_{\ff}(x^*\|x_{i,t})-D_{\ff}(x^*\|y_{i,t+1})-\f{1}{2}\|y_{i,t+1}-x_{i,t}\|^2
	\eea
	in which the first equality follows from the three point identity in Lemma \ref{bregman_3point} and the second inequality follows from the definition of $D_{\ff}(\cdot,\cdot)$ and $1$-strong convexity of $\ff$. Also, note that
	\bea
	\nono&&\left\nn\aaa_t\ww\nabla_{\lambda_t} f_i(x_{i,t}),y_{i,t+1}-x^*\right\mm\\
	\nono&&=\left\nn\aaa_t\ww\nabla_{\lambda_t} f_i(x_{i,t}),y_{i,t+1}-x_{i,t}\right\mm+\left\nn\aaa_t\ww\nabla_{\lambda_t} f_i(x_{i,t}),x_{i,t}-x^*\right\mm\\
	\nono&& \geq-\f{\aaa_t^2}{2}\left\|\ww\nabla_{\lambda_t} f_i(x_{i,t})\right\|^2-\f{1}{2}\|y_{i,t+1}-x_{i,t}\|^2\\
	&&\quad +\left\nn\aaa_t\ww\nabla_{\lambda_t} f_i(x_{i,t}),x_{i,t}-x^*\right\mm. \label{nei2}
	\eea
	Combining the above estimates, we obtain that
	\bes
	\nono&&\left\nn\aaa_t\ww\nabla_{\lambda_t} f_i(x_{i,t}),x_{i,t}-x^*\right\mm\\
	&&\leq \Big[D_{\ff}(x^*\|x_{i,t})-D_{\ff}(x^*\|y_{i,t+1})\Big]+\f{\aaa_t^2}{2}\left\|\ww\nabla_{\lambda_t} f_i(x_{i,t})\right\|^2.
	\ees
Then it follows that
\bes
\nono&&\left\nn\aaa_t\ww\nabla_{\lambda_t} f_i(x_{i,t}),x_{i,t}-x^*\right\mm\\
&&\leq \Big[D_{\ff}(x^*\|x_{i,t})-D_{\ff}(x^*\|x_{i,t+1})\Big]\\
&&+\Big[D_{\ff}(x^*\|x_{i,t+1})-D_{\ff}(x^*\|y_{i,t+1})\Big]\\
&&+\f{\aaa_t^2}{2}\left\|\ww\nabla_{\lambda_t} f_i(x_{i,t})\right\|^2.
\ees
After taking summation on both sides of the above inequality, we have
\bes
\nono&&\sum_{i=1}^m\left\nn\aaa_t\ww\nabla_{\lambda_t} f_i(x_{i,t}),x_{i,t}-x^*\right\mm\\
&&\leq \sum_{i=1}^m\Big[D_{\ff}(x^*\|x_{i,t})-D_{\ff}(x^*\|x_{i,t+1})\Big]\\
&&+\sum_{i=1}^m\Big[D_{\ff}(x^*\|x_{i,t+1})-D_{\ff}(x^*\|y_{i,t+1})\Big]\\
&&+\f{\aaa_t^2}{2}\sum_{i=1}^m\left\|\ww\nabla_{\lambda_t} f_i(x_{i,t})\right\|^2.
\ees
If $t+1\notin\huaI$, according to the  structure of our algorithm, $x_{i,t+1}=y_{i,t+1}$, $i\in[m]$, then 
$$D_{\ff}(x^*\|x_{i,t+1})-D_{\ff}(x^*\|y_{i,t+1})=0.$$
If $t+1\in\huaI$, $x_{i,t+1}=\bar y_{t+1}$, then, as a result of the separate convexity assumption of the Bregman divergence, we have
\bes
&&\sum_{i=1}^m\Big[D_{\ff}(x^*\|x_{i,t+1})-D_{\ff}(x^*\|y_{i,t+1})\Big]\\
&&=m\Big[D_{\ff}(x^*\|\bar y_{t+1})-\f{1}{m}D_{\ff}(x^*\|y_{i,t+1})\Big]\leq0.
\ees
Hence we arrive at 
$\sum_{i=1}^m\nn\aaa_t\ww\nabla_{\lambda_t} f_i(x_{i,t}),x_{i,t}-x^*\mm\leq \sum_{i=1}^m[D_{\ff}(x^*\|x_{i,t})-D_{\ff}(x^*\|x_{i,t+1})]+\f{\aaa_t^2}{2}\sum_{i=1}^m \|\ww\nabla_{\lambda_t} f_i(x_{i,t}) \|^2$.
We complete the proof.
\end{proof}

\begin{pro}\label{pro3}
Under Assumptions \ref{stochastic_gradient}-\ref{separable_convexity_ass}, denote a non-decreasing sequence $\{V_t\}$ as $$V_t=2m\max_{i\in[m], 1\leq s\leq t}\left\{\sqrt{2D_\ff(x^*\|x_{i,s})}\right\}+A$$ with some positive constant $A\geq8m$. If the stepsize $\aaa_t$ and the clipping parameter $\lambda_t$ satisfy $\aaa_t\lambda_t\leq1$, then for any  $\ell\in[m]$,  there holds, with probability at least $1-\delta$, for any $S\in[T]$,
	\bes
	&&\sum_{t=1}^S\f{\aaa_t}{V_t}\Big[ f(x_{\ell,t})-f(x^*)\Big]+\f{1}{V_S}\sum_{i=1}^mD_\ff(x^*\|x_{i,S+1})\\
	&&\leq\log\f{1}{\delta}+\f{1}{V_1}\sum_{i=1}^m D_\ff(x^*\|x_{i,1})+\sum_{t=1}^S\f{\aaa_t}{V_t}\sum_{i=1}^mL\|x_{\ell,t}-x_{i,t}\|^2\\
	&&\quad+\sum_{t=1}^S\f{\aaa_t}{V_t}\sum_{i=1}^m\|\nabla f_i(x_{i,t})\|\|x_{i,t}-x_{\ell,t}\|\\
	&&\quad+\sum_{t=1}^S\f{\aaa_t}{V_t}\sum_{i=1}^m\left\nn\ttt_{i,t}^b,x^*-x_{i,t}\right\mm+\sum_{t=1}^S\f{\aaa_t^2}{V_t}\sum_{i=1}^m\|\ttt_{i,t}^b\|^2\\
	&&\quad +\sum_{t=1}^S\left(\f{\aaa_t^2}{V_t}+3m\aaa_t^2+\f{6\aaa_t^4\lambda_t^2m}{V_t^2}\right)\sum_{i=1}^m\mbb E\Big[\|\ttt_{i,t}^u\|^2|\huaF_{t-1}\Big].
	\ees
\end{pro}
\begin{proof}
Recalling the concept $\ttt_{i,t}=\ww \nabla_{\lambda_t} f_i(x_{i,t})-\nabla f_i(x_{i,t})$, we know
$\ww \nabla_{\lambda_t} f_i(x_{i,t})=\ttt_{i,t}+\nabla f_i(x_{i,t}).$
Substituting this relation to Proposition \ref{inner_product_est1} and dividing both sides of the inequality by  $V_t$, we have
\bes
\nono&&\f{\aaa_t}{V_t}\sum_{i=1}^m\left\nn\nabla f_i(x_{i,t}),x_{i,t}-x^*\right\mm\\
&&\leq \sum_{i=1}^m\f{1}{V_t}\Big[D_{\ff}(x^*\|x_{i,t})-D_{\ff}(x^*\|x_{i,t+1})\Big]\\
&&\quad+\f{\aaa_t}{V_t}\sum_{i=1}^m\nn\ttt_{i,t},x^*-x_{i,t}\mm+\f{\aaa_t^2}{2V_t}\sum_{i=1}^m\left\|\ww\nabla_{\lambda_t} f_i(x_{i,t})\right\|^2.
\ees
Then based on the basic relation $\ttt_{i,t}=\ttt_{i,t}^u+\ttt_{i,t}^b$, we can further bound the above inequality as
\bes
\nono&&\f{\aaa_t}{V_t}\sum_{i=1}^m\left\nn\nabla f_i(x_{i,t}),x_{i,t}-x^*\right\mm\\
&&\leq \f{1}{V_t}\sum_{i=1}^m\Big[D_{\ff}(x^*\|x_{i,t})-D_{\ff}(x^*\|x_{i,t+1})\Big]\\
&&\quad+\f{\aaa_t}{V_t}\sum_{i=1}^m\nn\ttt_{i,t}^u,x^*-x_{i,t}\mm+\f{\aaa_t}{V_t}\sum_{i=1}^m\nn\ttt_{i,t}^b,x^*-x_{i,t}\mm\\
 &&\quad+\f{\aaa_t^2}{V_t}\sum_{i=1}^m\left[\|\ttt_{i,t}^u\|^2-\mbb E\Big[\|\ttt_{i,t}^u\|^2|\huaF_{t-1}\Big]\right]\\
 &&\quad+\f{\aaa_t^2}{V_t}\sum_{i=1}^m\mbb E\Big[\|\ttt_{i,t}^u\|^2|\huaF_{t-1}\Big]+\f{\aaa_t^2}{V_t}\sum_{i=1}^m\|\ttt_{i,t}^b\|^2.
\ees
We now set a random variable 
\bes
\beal
Z_t=&\f{\aaa_t}{V_t}\sum_{i=1}^m\left\nn\nabla f_i(x_{i,t}),x_{i,t}-x^*\right\mm\\
&-\f{1}{V_t}\sum_{i=1}^m\Big[D_{\ff}(x^*\|x_{i,t})-D_{\ff}(x^*\|x_{i,t+1})\Big]\\
&-\f{\aaa_t}{V_t}\sum_{i=1}^m\nn\ttt_{i,t}^b,x^*-x_{i,t}\mm-\f{\aaa_t^2}{V_t}\sum_{i=1}^m\|\ttt_{i,t}^b\|^2\\
&-\left(\f{\aaa_t^2}{V_t}+3m\aaa_t^2+\f{6\aaa_t^4\lambda_t^2m}{V_t^2}\right)\sum_{i=1}^m\mbb E\Big[\|\ttt_{i,t}^u\|^2|\huaF_{t-1}\Big].
\eeal
\ees
Then it holds that
\bea
\nono&&\hspace{-0.7cm}\mbb E\Big[\exp Z_t\Big|\huaF_{t-1}\Big]\\
\nono&&\hspace{-0.7cm}\times\exp\left(\left(3m\aaa_t^2+\f{6\aaa_t^4\lambda_t^2m}{V_t^2}\right)\sum_{i=1}^m\mbb E\Big[\|\ttt_{i,t}^u\|^2|\huaF_{t-1}\Big]\right)\\
\nono&&\hspace{-0.7cm}\leq\mbb E\Bigg[\exp\Bigg(\f{\aaa_t}{V_t}\sum_{i=1}^m\nn\ttt_{i,t}^u,x^*-x_{i,t}\mm\\
&&\hspace{-0.7cm}\quad+\f{\aaa_t^2}{V_t}\sum_{i=1}^m\left[\|\ttt_{i,t}^u\|^2-\mbb E\Big[\|\ttt_{i,t}^u\|^2|\huaF_{t-1}\Big]\right]\Bigg)\Bigg|\huaF_{t-1}\Bigg]. \label{eq1}
\eea
For the above two terms,  we know that 
\bes
&&\mbb E\left[\sum_{i=1}^m\nn\ttt_{i,t}^u,x^*-x_{i,t}\mm\right]\\
&&=\mbb E\left[\sum_{i=1}^m\left[\|\ttt_{i,t}^u\|^2-\mbb E\Big[\|\ttt_{i,t}^u\|^2|\huaF_{t-1}\Big]\right]\right]=0.
\ees
Recall that the 1-strong convexity of the mirror map $\ff$ indicates that $D_\ff(x^*\|x_{i,t})\geq\f{1}{2}\|x^*-x_{i,t}\|^2$. This fact together with Lemma \ref{bound_gradient_lemma} yields 
\bes
&&\left|\f{\aaa_t}{V_t}\sum_{i=1}^m\nn\ttt_{i,t}^u,x^*-x_{i,t}\mm\right|\leq\f{2\aaa_t\lambda_t}{V_t}\sum_{i=1}^m\|x^*-x_{i,t}\|\\
&&\leq\f{2\aaa_t\lambda_t}{V_t}\sum_{i=1}^m\sqrt{2D_\ff(x^*\|x_{i,t})}
\ees
and
\bes
\f{\aaa_t^2}{V_t}\sum_{i=1}^m\left[\|\ttt_{i,t}^u\|^2-\mbb E\Big[\|\ttt_{i,t}^u\|^2|\huaF_{t-1}\Big]\right]\leq\f{8m\aaa_t^2\lambda_t^2}{V_t}.
\ees
  Recalling the definition of $V_t$ with $A\geq8m$ and the condition $\aaa_t\lambda_t\leq1$, after applying Lemma \ref{random_ine_lemma} with $M=1$, we know 
\bes
&&\mbb E\Bigg[\exp\Bigg(\f{\aaa_t}{V_t}\sum_{i=1}^m\nn\ttt_{i,t}^u,x^*-x_{i,t}\mm\\
&&+\f{\aaa_t^2}{V_t}\sum_{i=1}^m\left[\|\ttt_{i,t}^u\|^2-\mbb E\Big[\|\ttt_{i,t}^u\|^2|\huaF_{t-1}\Big]\right]\Bigg)\Bigg|\huaF_{t-1}\Bigg]\\
&&\leq\exp\Bigg(\f{3}{4}\mbb E\Bigg[\bigg(\f{\aaa_t}{V_t}\sum_{i=1}^m\nn\ttt_{i,t}^u,x^*-x_{i,t}\mm\\
&&\quad+\f{\aaa_t^2}{V_t}\sum_{i=1}^m\left[\|\ttt_{i,t}^u\|^2-\mbb E\Big[\|\ttt_{i,t}^u\|^2|\huaF_{t-1}\Big]\right]\bigg)^2\Bigg|\huaF_{t-1}\Bigg]\Bigg)\\
&&\leq\exp\Bigg(\f{3}{4}\mbb E\Bigg[2\bigg(\f{\aaa_t}{V_t}\sum_{i=1}^m\nn\ttt_{i,t}^u,x^*-x_{i,t}\mm\bigg)^2\\
&&\quad+2\bigg(\f{\aaa_t^2}{V_t}\sum_{i=1}^m\left[\|\ttt_{i,t}^u\|^2-\mbb E\Big[\|\ttt_{i,t}^u\|^2|\huaF_{t-1}\Big]\right]\bigg)^2\Bigg|\huaF_{t-1}\Bigg]\Bigg).
\ees
Due to the fact that $V_t=\max_{i\in[m], 1\leq s\leq t}\left\{\sqrt{2D_\ff(x^*\|x_{i,s})}\right\}+A$ for some  $A>0$, and the strong convexity of the mirror map $\ff$ indicating that $D_\ff(x^*\|x_{i,t})\geq\f{1}{2}\|x^*-x_{i,t}\|^2$, we know that the above inequality can be bounded by
\bes
&&\exp\bigg(3m\aaa_t^2\sum_{i=1}^m\mbb E\Big[\|\ttt_{i,t}^u\|^2|\huaF_{t-1}\Big]\\
&&+\f{3\aaa_t^4m}{2\sss_\ff V_t^2}\sum_{i=1}^m\mbb E\Big[\|\ttt_{i,t}^u\|^4|\huaF_{t-1}\Big]\bigg)
\ees
which can be further bounded by 
\bes
\exp\left(\left(3m\aaa_t^2+\f{6\aaa_t^4\lambda_t^2m}{V_t^2}\right)\sum_{i=1}^m\mbb E\Big[\|\ttt_{i,t}^u\|^2|\huaF_{t-1}\Big]\right)
\ees
as a result of Lemma \ref{random_ine_lemma}. Then we have
$\mbb E\big[\exp Z_t\big|\huaF_{t-1}\big]\leq1$. Define $W_S=\sum_{t=1}^SZ_t$, then it follows that
\bes
\mbb E\Big[\exp W_S\Big|\huaF_{S-1}\Big]=(\exp W_{S-1})\mbb E\Big[\exp Z_S\Big|\huaF_{S-1}\Big]\leq\exp W_{S-1}.
\ees
Then we know $\{\exp W_{S}\}$ is a supermartingale. The Ville's inequality indicates that, for any $S\geq1$, $\mbb P(W_s\geq\log \f{1}{\delta})\leq\delta\mbb E[\exp W_1]$. Namely, with probability at least $1-\delta$, for any $S\geq1$, 
$$\sum_{t=1}^SZ_t\leq\log\f{1}{\delta}.$$
Then based on the above analysis, it holds that, with probability at least $1-\delta$,  for any $S\in[T]$, 
\bes
\hspace{-0.7cm}&&\sum_{t=1}^S\f{\aaa_t}{V_t}\sum_{i=1}^m\left\nn\nabla f_i(x_{i,t}),x_{i,t}-x^*\right\mm\\
\hspace{-0.7cm}&&\leq\log\f{1}{\delta}+\sum_{t=1}^S\f{1}{V_t}\sum_{i=1}^m\Big[D_{\ff}(x^*\|x_{i,t})-D_{\ff}(x^*\|x_{i,t+1})\Big]\\
\hspace{-0.7cm}&&\quad+\sum_{t=1}^S\f{\aaa_t}{V_t}\sum_{i=1}^m\nn\ttt_{i,t}^b,x^*-x_{i,t}\mm+\sum_{t=1}^S\f{\aaa_t^2}{V_t}\sum_{i=1}^m\|\ttt_{i,t}^b\|^2\\
\hspace{-0.7cm}&&\quad +\sum_{t=1}^S\left(\f{\aaa_t^2}{V_t}+3m\aaa_t^2+\f{6\aaa_t^4\lambda_t^2m}{V_t^2}\right)\sum_{i=1}^m\mbb E\Big[\|\ttt_{i,t}^u\|^2|\huaF_{t-1}\Big].
\ees
Note that
\bes
&&\sum_{t=1}^S\f{1}{V_t}\sum_{i=1}^m\Big[D_{\ff}(x^*\|x_{i,t})-D_{\ff}(x^*\|x_{i,t+1})\Big]\\
&&=\f{1}{V_1}\sum_{i=1}^m D_\ff(x^*\|x_{i,1})-\f{1}{V_S}\sum_{i=1}^mD_\ff(x^*\|x_{i,S+1})\\
&&\quad+\sum_{t=2}^S\left(\f{1}{V_t}-\f{1}{V_{t-1}}\right)\sum_{i=1}^mD_\ff(x^*\|x_{i,t+1}).
\ees
Since the sequence $\{V_t\}$ is non-decreasing, we know
$\f{1}{V_t}-\f{1}{V_{t-1}}\leq0$ for $t\geq2$. Then it follows that
\bes
&&\sum_{t=1}^S\f{1}{V_t}\sum_{i=1}^m\Big[D_{\ff}(x^*\|x_{i,t})-D_{\ff}(x^*\|x_{i,t+1})\Big]\\
&&\leq\f{1}{V_1}\sum_{i=1}^m D_\ff(x^*\|x_{i,1})-\f{1}{V_S}\sum_{i=1}^mD_\ff(x^*\|x_{i,S+1}).
\ees
On the other hand,  according to the convexity of $f_i$, we know that, for any $\ell\in[m]$,
\bea
\nono&&\hspace{-0.7cm}\left\nn\nabla f_i(x_{i,t}),x_{i,t}-x^*\right\mm\geq  f_i(x_{i,t})-f_i(x^*)\\
\nono&&\hspace{-0.7cm}=\Big[ f_i(x_{i,t})-f_i(x_{\ell,t})\Big]+\Big[ f_i(x_{\ell,t})-f_i(x^*)\Big]\\
&&\hspace{-0.7cm}\geq-\|\nabla f_i(x_{\ell,t})\|\|x_{i,t}-x_{\ell,t}\|+\Big[ f_i(x_{\ell,t})-f_i(x^*)\Big]. \label{inner_product_est}
\eea
According to the $L$-smoothness of $f_i$, we know
\bea
\beal
\|\nabla f_i(x_{\ell,t})\|\leq& \|\nabla f_i(x_{\ell,t})-\nabla f_i(x_{i,t})\|+\|\nabla f_i(x_{i,t})\|\\
\leq&L\|x_{\ell,t}-x_{i,t}\|+\|\nabla f_i(x_{i,t})\|. \label{Lsmooth_step}
\eeal
\eea
Therefore we have
\bes
&&\left\nn\nabla f_i(x_{i,t}),x_{i,t}-x^*\right\mm\geq-L\|x_{\ell,t}-x_{i,t}\|^2\\
&&-\|\nabla f_i(x_{i,t})\|\|x_{i,t}-x_{\ell,t}\|+\Big[ f_i(x_{\ell,t})-f_i(x^*)\Big].
\ees
Then it follows that, for any $\ell\in[m]$,  with probability at least $1-\delta$, for any $S\in[T]$, 
\bes
&&\sum_{t=1}^S\f{\aaa_t}{V_t}\sum_{i=1}^m\Big[ f_i(x_{\ell,t})-f_i(x^*)\Big]+\f{1}{V_S}\sum_{i=1}^mD_\ff(x^*\|x_{i,S+1})\\
&&\leq\log\f{1}{\delta}+\f{1}{V_1}\sum_{i=1}^m D_\ff(x^*\|x_{i,1})+\sum_{t=1}^S\f{\aaa_t}{V_t}\sum_{i=1}^mL\|x_{\ell,t}-x_{i,t}\|^2\\
&&\quad+\sum_{t=1}^S\f{\aaa_t}{V_t}\sum_{i=1}^m\|\nabla f_i(x_{i,t})\|\|x_{i,t}-x_{\ell,t}\|\\
&&\quad+\sum_{t=1}^S\f{\aaa_t}{V_t}\sum_{i=1}^m\left\nn\ttt_{i,t}^b,x^*-x_{i,t}\right\mm+\sum_{t=1}^S\f{\aaa_t^2}{V_t}\sum_{i=1}^m\|\ttt_{i,t}^b\|^2\\
&&\quad +\sum_{t=1}^S\left(\f{\aaa_t^2}{V_t}+3m\aaa_t^2+\f{6\aaa_t^4\lambda_t^2m}{V_t^2}\right)\sum_{i=1}^m\mbb E\Big[\|\ttt_{i,t}^u\|^2|\huaF_{t-1}\Big].
\ees
We complete the proof after noting that $f=\sum_{i\in[m]}f_i$.
\end{proof}

 For the sake of analysis, we assume that we have access to two problem parameters. One is the maximum of the initial Bregman distance of the local state variables to the optimum   given as $R_1=\max_{i\in[m]}\{\sqrt{2D_\ff(x^*\|x_{i,1})}\}$. In theory and practice, we only need to estimate a rough upper bound for $R_1$ (see e.g. \cite{gsdhgdgr2024}), for which we still denote this rough upper bound as $R_1$. In many practical problems, especially those involving bounded decision regions with known diameters, such as Euclidean ball, Euclidean cube or probability simplex, the value of $R_1$
can be easily estimated.
We also denote the maximum of the local initial gradient estimate of the local objective functions at the initial point by
$B=\max_{i\in[m]}\{\|\nabla f_i(x_{i,1})\|\}$.
As pointed out in \cite{nnen2023}, \cite{gsdhgdgr2024}, there are already mature techniques such as the tools in \cite{minsker2015} for estimating the initial gradient via stochastic gradient samples and their geometric median. Thus the consideration of problem parameter $B$ is meaningful in the existing works and subsequent theoretical analysis.

In the following, we use $E(\delta)$ to denote the event that, for any $S\in[T]$,  the inequality in Proposition \ref{pro3} holds.  Now we are ready to present the following main result. The result is the core foundation of the high probability convergence theory developed in this paper. It has provided the inductive proof framework for the subsequent convergence proofs of Clipped FedSMD.
\begin{thm}\label{thm1}
Suppose that Assumptions \ref{stochastic_gradient}-\ref{random_ine_lemma} hold and the event $E(\delta)$ happens.  For any fixed $0<\delta<1$, suppose that  $A$ given in Proposition \ref{pro3} takes the form 
\bes
\beal
A=&\log\f{1}{\delta}+mL\huaP^2 C_0+2m\huaP C_1+4\sigma^pC_2\\
&+16m\sigma^{2p}C_3+40m\sss^p(1+3m)C_4+40m^2\sss^pC_5+8m,
\eeal
\ees
and the stepsize $\aaa_t$ and the clipping parameter $\lambda_t$ satisfy 
\bes
&&\hspace{-0.8cm}\aaa_t=\f{1}{(1+\log t)^{\gamma}t^{\kkk-\mu}}\min\left\{\f{1}{t^\mu},[2mL(2R_1+10A)+2B]^{-1}\right\}\\ 
&&\hspace{-0.8cm}\lambda_t=\max\left\{t^\mu, 2mL(2R_1+10A)+2B\right\} 
\ees
with the parameters $0<\kkk<1$ and $0<\mu<1$ satisfying
\bes
\kkk\geq\max\{\mu+\f{1}{2},1-\mu(p-1)\},  
\ees
$\gamma>1$ being a parameter larger than $1$ and the finite constants $C_0$-$C_5$ given as
$C_0=\sum_{t=1}^\infty\aaa_t\aaa_{\tau(t)}^2\lambda_{\tau(t)}^2$, $C_1=\sum_{t=1}^\infty\aaa_t\aaa_{\tau(t)}\lambda_t\lambda_{\tau(t)}$, $C_2=\sum_{t=1}^\infty\aaa_t\lambda_t^{1-p}$,
$C_3=\sum_{t=1}^\infty\aaa_t^2\lambda_t^{2-2p}$, $C_4=\sum_{t=1}^\infty(\aaa_t\lambda_t)^2(\f{1}{\lambda_t})^p$,
$C_5=\sum_{t=1}^\infty(\aaa_t\lambda_t)^4(\f{1}{\lambda_t})^p$.
If for some $K\in[T]$,   there holds that  $$\|\nabla f_i(x_{i,t})\|\leq\f{\lambda_t}{2}, i\in[m], 1\leq t\leq K-1,$$ then for any $\ell\in[m]$ and $1\leq S\leq K$, it holds that
	\bes
	&&\sum_{t=1}^S\aaa_t\Big[ f(x_{\ell,t})-f(x^*)\Big]+\sum_{i=1}^mD_\ff(x^*\|x_{i,S+1})\\
	&&\leq m^2(R_1+10A)^2.
	\ees
\end{thm}
\begin{proof}
Under $\|\nabla f_i(x_{i,t})\|\leq\f{\lambda_t}{2}$, $i\in[m]$, according to Lemma \ref{random_ine_lemma},  we know 
\bes
\beal
|\nn\ttt_{i,t}^b,x_{i,t}-x^*\mm|\leq&\|\ttt_{i,t}^b\|\|x_{i,t}-x^*\|\\
\leq&4\sss^p\lambda_t^{1-p}\sqrt{2D_\ff(x^*\|x_{i,t})};\\
\|\ttt_{i,t}^b\|\leq&4\sss^p\lambda_t^{1-p};\\
\mbb E[\|\ttt_{i,t}^u\|^2|\huaF_{t-1}]\leq&40\sigma^p\lambda_t^{2-p}.
\eeal
\ees
Next, we use induction to prove the statement. When $S=0$, it is trivial to see the result holds. Now  suppose that, when $1\leq k\leq S-1$, there holds, 
\bes
&&\sum_{t=1}^{k}\aaa_t\sum_{i=1}^m\Big[ f_i(x_{\ell,t})-f_i(x^*)\Big]+\sum_{i=1}^mD_\ff(x^*\|x_{i,k+1})\\
&&\leq m^2(R_1+10A)^2.
\ees
Then it follows that
\bes
\sum_{i=1}^mD_\ff(x^*\|x_{i,S})\leq m^2(R_1+10A)^2.
\ees
Hence we have
\bea
\max_{i\in[m]}\sqrt{D_\ff(x^*\|x_{i,S})}\leq m(R_1+10A). \label{induction_ass}
\eea
When $k=S$, According to Proposition \ref{pro3}, it holds that, for any $\ell\in[m]$, 
\bes
&&\sum_{t=1}^S\f{\aaa_t}{V_t}\Big[ f(x_{\ell,t})-f(x^*)\Big]+\f{1}{V_S}\sum_{i=1}^mD_\ff(x^*\|x_{i,S+1})\\
&&\leq\log\f{1}{\delta}+\f{1}{V_1}\sum_{i=1}^m D_\ff(x^*\|x_{i,1})+\sum_{t=1}^S\f{\aaa_t}{V_t}\sum_{i=1}^mL\|x_{\ell,t}-x_{i,t}\|^2\\
&&\quad+\sum_{t=1}^S\f{\aaa_t}{V_t}\sum_{i=1}^m\|\nabla f_i(x_{i,t})\|\|x_{i,t}-x_{\ell,t}\|\\
&&\quad+\sum_{t=1}^S\f{\aaa_t}{V_t}\sum_{i=1}^m\left\nn\ttt_{i,t}^b,x^*-x_{i,t}\right\mm+\sum_{t=1}^S\f{\aaa_t^2}{V_t}\sum_{i=1}^m\|\ttt_{i,t}^b\|^2\\
&&\quad +\sum_{t=1}^S\left(\f{\aaa_t^2}{V_t}+3m\aaa_t^2+\f{6\aaa_t^4\lambda_t^2m}{V_t^2}\right)\sum_{i=1}^m\mbb E\Big[\|\ttt_{i,t}^u\|^2|\huaF_{t-1}\Big].
\ees
Note that, according to Proposition \ref{basic_consensus}, we have
\bes
&&\sum_{t=1}^S\f{\aaa_t}{V_t}\sum_{i=1}^mL\|x_{\ell,t}-x_{i,t}\|^2\\
&&\leq \f{mL\huaP^2}{2m R_1+A}\sum_{t=1}^S\aaa_t\aaa_{\tau(t)}^2\lambda_{\tau(t)}^2\leq mL\huaP^2 C_0,
\ees
and
\bea
\nono&&\sum_{t=1}^S\f{\aaa_t}{V_t}\sum_{i=1}^m\|\nabla f_i(x_{i,t})\|\|x_{i,t}-x_{\ell,t}\|\\
\nono&&\leq 2m\sum_{t=1}^S\f{\aaa_t\lambda_t}{V_t}\sum_{s=\tau(t)}^{t-1}\aaa_s\lambda_s\leq \f{2m\huaP}{V_1}\sum_{t=1}^S\aaa_t\aaa_{\tau(t)}\lambda_t\lambda_{\tau(t)}\\
&&\leq\f{2m\huaP}{2m R_1+A}C_1.  \label{C_1bdd}
\eea
Meanwhile, the Schwarz inequality and the strong convexity of $\ff$ indicate that
\bes
&&\sum_{t=1}^S\f{\aaa_t}{V_t}\sum_{i=1}^m\left|\left\nn\ttt_{i,t}^b,x^*-x_{i,t}\right\mm\right|\leq\sum_{t=1}^S\f{\aaa_t}{V_t}\sum_{i=1}^m\|\ttt_{i,t}^b\|\|x_{i,t}-x^*\|\\
&&\leq4\sum_{t=1}^S\f{\aaa_t}{V_t}\sum_{i=1}^m\sss^p\lambda_t^{1-p}\sqrt{2D_\ff(x^*\|x_{i,t})}\\
&&\leq4\sss^p\sum_{t=1}^S\aaa_t\lambda_t^{1-p}\leq4\sss^pC_2.
\ees
Also we have
\bes
\sum_{t=1}^S\f{\aaa_t^2}{V_t}\sum_{i=1}^m\|\ttt_{i,t}^b\|^2\leq\f{16m\sss^{2p}}{R_1+A}\sum_{t=1}^S\aaa_t^2\lambda_t^{2-2p}\leq\f{16m\sss^{2p}}{2mR_1+A}C_3.
\ees
We also note that
\bes
&&\sum_{t=1}^S\left(\f{\aaa_t^2}{V_t}+3m\aaa_t^2+\f{6\aaa_t^4\lambda_t^2m}{V_t^2}\right)\sum_{i=1}^m\mbb E\Big[\|\ttt_{i,t}^u\|^2|\huaF_{t-1}\Big]\\
&&\leq40m\sss^p\left(\f{1}{V_1}+3m\right)\sum_{t=1}^S(\aaa_t\lambda_t)^2(\f{1}{\lambda_t})^p\\
&&\quad+\f{40m^2\sss^p}{ V_1^2}\sum_{t=1}^S(\aaa_t\lambda_t)^4(\f{1}{\lambda_t})^p\\
&&\leq40m\sss^p\left(\f{1}{V_1}+3m\right)C_4+\f{40m^2\sss^p}{ V_1^2}C_5.
\ees
The finiteness of the constants $C_0$-$C_5$ is ensured by the conditions on $\kkk$ and $\mu$ of the theorem and basic calculus. We briefly show them as follows. For $C_0$ and $C_1$, it is easy to see $C_0,C_1\leq\sum_{t=1}^\infty\aaa_{\tau(t)}^2\lambda_{\tau(t)}^2$. Then according to the definition of $\tau(t)$, we know
\bes
\beal
\sum_{t=1}^\infty\aaa_{\tau(t)}^2\lambda_{\tau(t)}^2=&\lim_{\huaT\rightarrow\infty}\huaP\sum_{k=0}^\huaT\aaa_{1+k\huaP}^2\lambda_{1+k\huaP}^2\\
=&\huaP\sum_{k=0}^\infty\f{1}{(1+\log (1+k\huaP))^{2\gamma}(1+k\huaP)^{2(\kkk-\mu)}}\\
\leq&\huaP\sum_{k=0}^\infty\f{1}{(1+\log(1+k))^{2\gamma}(1+k)^{2(\kkk-\mu)}}.
\eeal
\ees
Since $\gamma>1$, $\kkk-\mu\geq\f{1}{2}$, we know $C_0<\infty$, and $C_1<\infty$. Because $C_2\geq C_3$ and $C_2\geq C_4\geq C_5$, we only need to verify $C_2<\infty$. Note that
\bes
&&C_2=\sum_{t=1}^\infty\aaa_t\lambda_t^{1-p}=\sum_{t=1}^\infty\f{1}{(1+\log t)^\gamma t^{\kkk-\mu}}\f{1}{\lambda_t^p}\\
&&\leq \sum_{t=1}^\infty\f{1}{(1+\log t)^\ttt t^{\kkk-\mu}}\f{1}{t^{p\mu}}\leq \sum_{t=1}^\infty\f{1}{(1+\log t)^\gamma t^{\kkk+(p-1)\mu}},
\ees
and the selection of $\kkk$ satisfies $\kkk\geq1-\mu(p-1)$, we know $C_2$ is finite.
 
 Combining the above estimates and noticing the selection rule of $A$, we arrive at
\bes
&&\sum_{t=1}^S\f{\aaa_t}{V_t}\Big[ f(x_{\ell,t})-f(x^*)\Big]+\f{1}{V_S}\sum_{i=1}^mD_\ff(x^*\|x_{i,S+1})\\
&&\leq\f{m R_1^2}{2mR_1+A}+\log\f{1}{\delta}+\f{m\huaP}{2mR_1+A}C_1+4\sss^pC_2\\
&&\quad+\f{16m\sss^{2p}}{2mR_1+A}C_3+40m\sss^p\left(\f{1}{V_1}+3m\right)C_4+\f{40m^2\sss^p}{ V_1^2}C_5\\
&&\leq \f{m R_1^2}{2mR_1+A}+A.
\ees
Due to the induction assumption, we have known from \eqref{induction_ass} that $V_S\leq2\sqrt{2}m^2A(R_1+10A)+A$, this together with the non-decreasing property of $\{V_t\}$ yields that
\bes
&&\sum_{t=1}^S\aaa_t\Big[ f(x_{\ell,t})-f(x^*)\Big]+\sum_{i=1}^mD_\ff(x^*\|x_{i,S+1})\\
&&\leq \left(\f{m R_1^2}{2mR_1+A}+A\right)[2\sqrt{2}m^2A(R_1+10A)+A]\\
&&\leq m^2(R_1+10A)^2.
\ees
We complete the proof.
\end{proof}

 For any agent $\ell\in[m]$, we consider the following local ergodic approximating sequence $$\widehat{x}_\ell^{T}:=\f{1}{T}\sum_{t=1}^{T}x_{\ell,t}.$$ Based on the results in Theorem \ref{thm1}, the next theorem provides a general convergence bound for Clipped FedSMD.
\begin{thm}\label{thm2}
Under Assumptions \ref{stochastic_gradient}-\ref{random_ine_lemma},  for any fixed $0<\delta<1$, if the stepsize $\aaa_t$ and the clipping parameter $\lambda_t$ satisfy the conditions of Theorem \ref{thm1}, then for any agent $\ell\in[m]$,  we have, with probability at least $1-\delta$, for any $T\in\mbb N$,
	 \bes
	\beal
	f(\wh x_\ell^T)-f(x^*)\leq\f{1}{T\aaa_T}[m^2(R_1+10A)^2].
	\eeal
	\ees
\end{thm}
\begin{proof} 
	We prove by induction that, for all $t$, there holds that
 \bes
 \|\nabla f_i(x_{i,t})\|\leq\f{\lambda_t}{2}, \ i\in[m].
 \ees
 For $t=1$, it is easy to see
 \bes
 \|\nabla f_i(x_{i,1})\|\leq B\leq\f{\lambda_1}{2}, \ i\in[m].
 \ees
Now suppose that for all $t\leq S$, there holds $\|\nabla f_i(x_{i,t})\|\leq\f{\lambda_t}{2}$, $i\in[m]$. We show that
 \bes
 \|\nabla f_i(x_{i,S+1})\|\leq\f{\lambda_{S+1}}{2}, \ i\in[m].
 \ees
 According to Theorem \ref{thm1}, under condition $\|\nabla f_i(x_{i,t})\|\leq\f{\lambda_t}{2}$, it holds that 
 \bes
 \|x_{i,S+1}-x^*\|\leq\sqrt{2D_\ff(x^*\|x_{i,S+1})}\leq m(R_1+10A).
 \ees
 Then it follows from the selection rule of $\lambda_t$ in Theorem \ref{thm1} that, for any $i\in[m]$,
 \bes
\beal
\|\nabla f_i(x_{i,S+1})\|\leq&\|\nabla f_i(x_{i,S+1})-\nabla f_i(x^*)\|\\
&+\|\nabla f_i(x^*)-\nabla f_i(x_{i,1})\|+\|\nabla f_i(x_{i,1})\|\\
\leq&L\|x_{i,S+1}-x^*\|+L\|x^*-x_{i,1}\|+B\\
\leq&mL(R_1+10A)+LR_1+B\leq\f{\lambda_t}{2}.
\eeal
\ees
Therefore, for all $t$, $ \|\nabla f_i(x_{i,t})\|\leq\f{\lambda_t}{2}, \ i\in[m]$. Then according to Theorem \ref{thm1}, we have, for any agent $\ell\in[m]$,  with probability at least $1-\delta$, for any $T\in\mbb N$,
\bes
&&\sum_{t=1}^T\aaa_t\Big[ f(x_{\ell,t})-f(x^*)\Big]+\sum_{i=1}^mD_\ff(x^*\|x_{i,T+1})\\
&&\leq m^2(R_1+10A)^2.
\ees
According to the positivity of $D(x^*\|x_{i,T+1})$, $i\in[m]$ and the fact that $\aaa_t$ is  non-increasing, we have, with probability at least $1-\delta$, for any $T\in\mbb N$,
$$\sum_{t=1}^T [ f(x_{\ell,t})-f(x^*) ]\leq \f{1}{\aaa_T}[m^2(R_1+10A)^2].$$
We finish the proof after dividing both sides of the above inequality by $T$ and utilizing the convexity of $f$.
\end{proof}

 The above theorem has achieved a general high probability convergence bound under the selection of stepsizes and clipping parameters satisfying the relating $\kkk\geq\max\{\mu+\f{1}{2},1-\mu(p-1)\}$. When $(\kkk,\mu)$ corresponds to the minmax pair that satisfies $\kkk\geq\max\{\mu+\f{1}{2},1-\mu(p-1)\}$, namely $(\kkk,\mu)=(\f{p+1}{2p},\f{1}{2p})$, the next result provides the corresponding  explicit high probability convergence rate in terms of the total iteration $T$.

\begin{thm}\label{thm3}
	Under Assumptions \ref{stochastic_gradient}-\ref{random_ine_lemma},  for any fixed $0<\delta<1$, if the stepsize $\aaa_t$ and the clipping parameter $\lambda_t$ satisfy the condition of Theorem \ref{thm1}, and $(\kkk,\mu)=(\f{p+1}{2p},\f{1}{2p})$, 
	 then for any $\ell\in[m]$, we have, with probability at least $1-\delta$, 
		\bes
	f(\wh x_\ell^T)-f(x^*)\leq\huaO\left(T^{\f{1-p}{2p}}\log^{\gamma}T\log^2\left(\f{1}{\delta}\right)\right),
	\ees
	for sufficiently large $T$.
	\end{thm}
\begin{proof}
 After dividing both sides by $T$ and utilizing the convexity of $f$, we have, with probability at least $1-\delta$,
\bes
\beal
f(\wh x_\ell^T)-f(x^*)\leq\f{1}{T\aaa_T}[m^2(R_1+10A)^2].
\eeal
\ees
After directly substituting the values of $A$, $\aaa_t$, $\lambda_t$ with the parameter values $(\kkk,\mu)=(\f{p+1}{2p},\f{1}{2p})$ into Theorem \ref{thm2}, we  obtain the desired high probability convergence rate with sufficiently large $T$.
\end{proof}

Next, we will discuss a commonly used scenario and present the corresponding main results. In distributed optimization, in addition to Assumption \ref{Lsmooth_condition},  another widely considered assumption is the following boundedness condition of the gradient.
\begin{ass}\label{gradient_bdd_ass}
	There exists a constant $G>0$, such that, for any $i\in[m]$, $\|\nabla f_i(x)\|\leq G$, $x\in\huaX$.
\end{ass}
Assumption \ref{gradient_bdd_ass} are easily satisfied in conventional optimization scenarios, such as smooth objective functions defined in common regions like the Euclidean ball and simplex. However, it is worth mentioning that, we still do not require a bounded closed region (hence compact) condition with an explicit diameter of the decision domain as adopted in many existing works (e.g., \cite{yhy2022}, \cite{xzhyx2022}, \cite{qlxc2025}, \cite{ylw2025}) to achieve the corresponding high-probability convergence rate. This has clear theoretical value.

Next, we will provide selection rules for the stepsizes and clipping parameters after replacing Assumption  \ref{Lsmooth_condition} with Assumption \ref{gradient_bdd_ass}, along with the corresponding convergence results of Clipped FedSMD based on our established results. We show that the same convergence rates can also be derived.

\begin{thm}\label{thm4}
	Under Assumptions \ref{stochastic_gradient}, \ref{random_ine_lemma} and \ref{gradient_bdd_ass},  for any fixed $0<\delta<1$, if the stepsize $\aaa_t$ and the clipping parameter $\lambda_t$ satisfy 
	\bes
	&&\hspace{-0.8cm}\aaa_t=\f{1}{(1+\log t)^{\gamma}t^{\kkk-\mu}}\min\left\{\f{1}{t^\mu},[2G]^{-1}\right\}\\ 
	&&\hspace{-0.8cm}\lambda_t=\max\left\{t^\mu, 2G\right\} 
	\ees
	with the parameters $0<\kkk<1$ and $0<\mu<1$ satisfying
	\bes
	\kkk\geq\max\{\mu+\f{1}{2},1-\mu(p-1)\},  
	\ees 
	and the parameter $\gamma$ satisfying $\gamma>1$, then for any agent $\ell\in[m]$,  we have, with probability at least $1-\delta$, for any $T\in\mbb N$,
	\bes
	\beal
	f(\wh x_\ell^T)-f(x^*)\leq\f{1}{T\aaa_T}[m^2(R_1+10A)^2].
	\eeal
	\ees
	Here, $A$ is defined as in Theorem \ref{thm1} with $C_0=0$ and $C_1$-$C_5$ being defined in Theorem \ref{thm1}.
	\end{thm}
\begin{proof}
The core proof follows the same procedures from Proposition \ref{pro3} to Theorem \ref{thm1}. We only mention the difference here. The only difference is the step from \eqref{inner_product_est} to \eqref{Lsmooth_step}. Based on Assumption \ref{gradient_bdd_ass},
for any $\ell\in[m]$, \eqref{inner_product_est} to \eqref{Lsmooth_step} will be estimated as follow
\bea
\nono&&\hspace{-0.7cm}\left\nn\nabla f_i(x_{i,t}),x_{i,t}-x^*\right\mm\geq  f_i(x_{i,t})-f_i(x^*)\\
\nono&&\hspace{-0.7cm}=\Big[ f_i(x_{i,t})-f_i(x_{\ell,t})\Big]+\Big[ f_i(x_{\ell,t})-f_i(x^*)\Big]\\
&&\hspace{-0.7cm}\geq-G\|x_{i,t}-x_{\ell,t}\|+\Big[ f_i(x_{\ell,t})-f_i(x^*)\Big]. \label{inner_product_estthm4}
\eea
Accordingly, for any $\ell\in[m]$, due to the selection of $\lambda_t=\max\left\{t^\mu, 2G\right\}$ and $\|\nabla f_i\|\leq G$, it holds that $\|\nabla f_i(x_{\ell,t})\|\leq\f{\lambda_t}{2}$, and we know $\sum_{t=1}^S\f{\aaa_t}{V_t}\sum_{i=1}^m\|\nabla f_i(x_{\ell,t})\|\|x_{i,t}-x_{\ell,t}\|$ can be bounded in a similar way with \eqref{C_1bdd} as 
\bes
\nono&&\sum_{t=1}^S\f{\aaa_t}{V_t}\sum_{i=1}^m\|\nabla f_i(x_{\ell,t})\|\|x_{i,t}-x_{\ell,t}\|\\
\nono&&\leq 2m\sum_{t=1}^S\f{\aaa_t\lambda_t}{V_t}\sum_{s=\tau(t)}^{t-1}\aaa_s\lambda_s\leq \f{2m\huaP}{V_1}\sum_{t=1}^S\aaa_t\aaa_{\tau(t)}\lambda_t\lambda_{\tau(t)}\\
&&\leq\f{2m\huaP}{2m R_1+A}C_1.
\ees
Also, under Assumption \ref{gradient_bdd_ass}, $\|\nabla f_i\|\leq G$, based on the selection of $\lambda_t=\max\left\{t^\mu, 2G\right\}$, it is easy to see, for any $i\in[m]$, $\|\nabla f_i(x_{i,t})\|\leq \f{\lambda_t}{2}$. Hence we can follow the same procedures of Theorem \ref{thm1} to obtain the desired result.
\end{proof}

Equipped with the above analysis, we can directly state the following result on the convergence rate of Clipped FedSMD under bounded gradient condition.
\begin{cor}\label{thm5}
Under Assumptions \ref{stochastic_gradient}, \ref{random_ine_lemma} and \ref{gradient_bdd_ass},  for any fixed $0<\delta<1$, if the stepsize $\aaa_t$ and the clipping parameter $\lambda_t$ satisfy the condition of Theorem \ref{thm4}, and $(\kkk,\mu)=(\f{p+1}{2p},\f{1}{2p})$, 
then for any $\ell\in[m]$, we have, with probability at least $1-\delta$, 
\bes
f(\wh x_\ell^T)-f(x^*)\leq\huaO\left(T^{\f{1-p}{2p}}\log^{\gamma}T\log^2\left(\f{1}{\delta}\right)\right),
\ees
for sufficiently large $T$.
\end{cor}

\begin{rmk}
	It is noteworthy that, in DSO, existing convergence results for algorithms under heavy-tailed noise are scarce compared to those under light-tailed noise. In particular, to our knowledge, this paper provides a comprehensive theoretical foundation and convergence rate analysis for the federated SMD algorithm under heavy-tailed noise environments for the first time in the context of federated optimization/learning. Our high probability convergence rates are derived in terms of tail parameter $p\in(1,2]$ that is related to the heavy-tail stochastic gradient  and an arbitrary $\gamma>1$ depending on the selection of the stepsize $\aaa_t$. The above convergence results are obtained by developing a set of inductive proof techniques suitable for distributed/federated algorithms. By deriving convergence rates under two fundamental assumptions, this paper enriches the study of DSO problems with heavy-tailed stochastic gradient noise.
 \end{rmk}

\begin{rmk}
	Theorems \ref{thm1}-\ref{thm3} are established  under the $L$-smoothness condition of the local objective functions, compared to the $G$-gradient boundedness condition considered in Theorem \ref{thm4}-Corollary \ref{thm5}. Accordingly, this fact results in a different constant selections in the stepsize $\aaa_t$ and clipping parameter $\lambda_t$ between them. We have highlighted the technical differences in detail within the proof of Theorem \ref{thm4}, compared with Theorem \ref{thm1}-\ref{thm3}. We can also see that directly assuming gradient bounds in Theorem \ref{thm4}-Corollary \ref{thm5} relatively simplifies the parameter selection of $\aaa_t$, $\lambda_t$ and the proof process, compared to Theorems \ref{thm1}-\ref{thm3}. In practice, choosing the right theorem to apply can depend on the specific smoothness and boundedness of the local objective functions $f_i, i\in[m]$. 
\end{rmk}

We conclude this section by presenting two typical scenarios, along with the established main algorithm.  When the mirror map $\Phi$ is taken as $\f{1}{2}\|x\|^2$, our local update of (2) in Algorithm 1 reduces to the well-known projected stochastic gradient descent (SGD) structure, namely, 
\bes
	y_{i,t+1}=P_\huaX\left(x_{i,t}-\aaa_t\ww \nabla_{\lambda_t} f_i(x_{i,t})\right).
\ees
 where $P_\huaX$ denotes the projection operator $P_\huaX(\cdot)=\arg\min_{x\in\huaX}\|\cdot-x\|$. Accordingly, the whole algorithm becomes a Clipped FedSGD.
If we consider the probability simplex decision space $\huaX=\Delta_n=\{x\in\mbb R^n:\sum_{i=1}^n[x]_i=1, [x]_i\geq0, i\in [n]\}$, we utilize the Kullback-Leibler divergence $D_\ff(x\|y)=\sum_{i=1}^n[x]_i\ln\f{[x]_i}{[y]_i}$ induced by the Gibbs entropy function $\Phi(x)=\sum_{i=1}^n[x]_i\ln[x]_i$. Then the corresponding main algorithm reduces to a Clipped Federated stochastic entropic descent (Clipped FedSED) with the local update in (2) of Algorithm 1 becoming
\bes
[y_{i,t+1}]_j=\f{[x_{i,t}]_j\exp(-\aaa_t[\ww \nabla_{\lambda_t} f_i(x_{i,t})]_j)}{\sum_{s=1}^n[x_{i,t}]_s\exp(-\aaa_t[\ww \nabla_{\lambda_t} f_i(x_{i,t})]_s)}, \ j\in[n].
\ees
Both types of algorithms are of  importance. With the convergence theory developed in this paper, the convergence performance of these two algorithms can naturally be guaranteed.

\section{Simulation experiments}
This section verifies the convergence of the algorithm Clipped FedSMD through the following linear regression problem:
\begin{align*}
\min_{x\in\mathcal{X}}\sum_{i=1}^m f_i(x):=\sum_{i=1}^m\frac{1}{2}(\langle a_i,x\rangle-b_i)^2,
\end{align*}
where $\{a_i\in\mathbb{R}^n\}_{i=1}^m$ is the sequence of feature vectors, and $\{b_i\in\mathbb{R}\}_{i=1}^m$ is the sequence of target values.

In the following simulations, each element of $a_i$ is uniformly distributed over the interval $[-1,1]$, and $b_i=\langle a_i,c\rangle+\epsilon_i$, where
\begin{align*}
[c]_i=\begin{cases}
1,&~1\leq i\leq\lfloor\frac{n}{2}\rfloor,\\
0,&~\lfloor\frac{n}{2}\rfloor<i\leq n,
\end{cases}
\end{align*}
and $\epsilon_i$ is independently and identically distributed according to the normal distribution $\mathcal{N}(0,1)$. The gradient noise $\xi_{i,t}:=\nabla f_i(x_{i,t})-\widehat{\nabla} f_i(x_{i,t})$ is generated via $\xi_{i,t}=\eta_{i,t}-\mbb E[\eta_{i,t}]$, where the probability density function of any components of $\eta_{i,t}$, $i\in[m]$ is given via the Pareto probability density function
\begin{align*}
h(x)=
\begin{cases}
\frac{\beta x_s^\beta}{x^{\beta+1}},&~x\geq x_s,\\
0,&~x<x_s,
\end{cases}
\end{align*}
with $\beta=2$, $x_s=0.5$.

 The decision set $\mathcal{X}$ is configured as the probability simplex $\{x\in\mathbb{R}^2:\sum_{i=1}^2[x]_i=1,[x]_i\geq0,i\in[2]\}$. Additionally, the distance generation function is set to the entropic function $\Phi(x)=\sum_{i=1}^2[x]_i\ln[x]_i$.
Furthermore, with a communication period of $\mathcal{P}=2$ and communication rounds of $\mathcal{T}=30000$, the total number of iterations is correspondingly $T=60001$. The selections of the stepsize $\alpha_t$ and the clipping parameter $\lambda_t$ follow from the rules of Theorem 4, with the corresponding parameters $\mu=0.5/p$, $\kappa=\mu+0.5$, $\gamma=1.01$ with the tail parameter $p=1.8$ related to the heavy-tailed noise.

For convenience, the global average optimization error is denoted as $\mathsf{E}$, with its specific expressions given as
\begin{align*}
&\mathsf{E}=\frac{1}{m}\sum_{\ell=1}^mf(\widehat{x}_{\ell}^T)-f(x^*).
\end{align*}
Based on the above setting, we conducted a series of simulation experiments to evaluate the performance of the Clipped FedSMD.
In each experiment, only the specific parameters under comparison were varied, while all other parameters were held constant.

As shown in Fig. \ref{comp_m}, the global average optimization error increases with the number of clients $m$, indicating a slower convergence performance with increasing number of agents.  Then, the impact of different communication periods on the convergence rate was compared. According to Fig. \ref{comp_ps}, as the communication period $\mathcal{P}$ increases, the average global optimization error of the algorithm also increases. This is as expected, because for the same number of iterations, a larger communication period means each client receives global information from the server less frequently, resulting in a slower convergence performance. Finally, the effect of the tail parameter $p$ (that is used to measure the heavy-tailed randomness of the stochastic gradient noises) on the algorithm's performance was compared. As shown in Fig. \ref{comp_p}, a larger tail parameter $p$ results in a smaller global average optimization error, which indicates that the algorithm converges faster. This provides an intuitive validation that the stronger the heavy-tailed effect, the greater the impact on the corresponding convergence performance.

\begin{figure}[htbp]
	\centerline{\includegraphics[width=0.8\columnwidth]{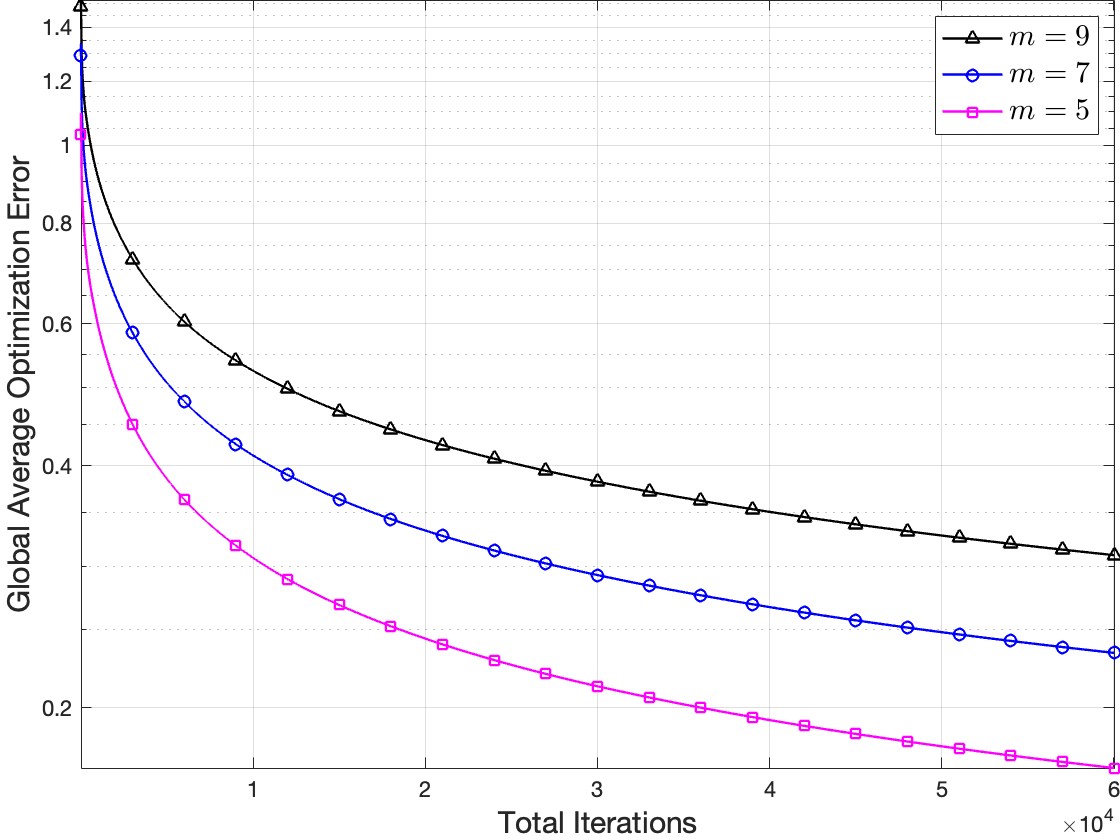}}
	\caption{Global average optimization error under different numbers $m$ of clients.}
	\label{comp_m}
\end{figure}
\begin{figure}[htbp]
	\centerline{\includegraphics[width=0.8\columnwidth]{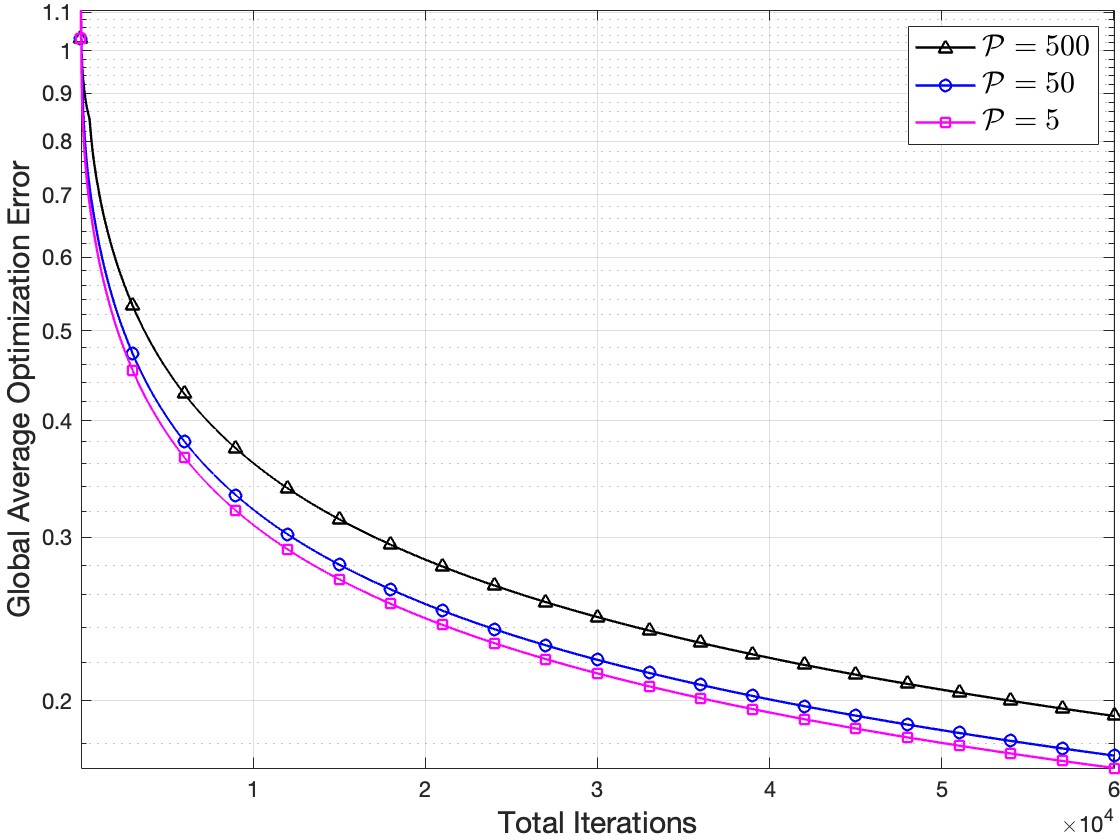}}
	\caption{Global average optimization error under different communication periods $\huaP$.}
	\label{comp_ps}
\end{figure}
\begin{figure}[htbp]
	\centerline{\includegraphics[width=0.8\columnwidth]{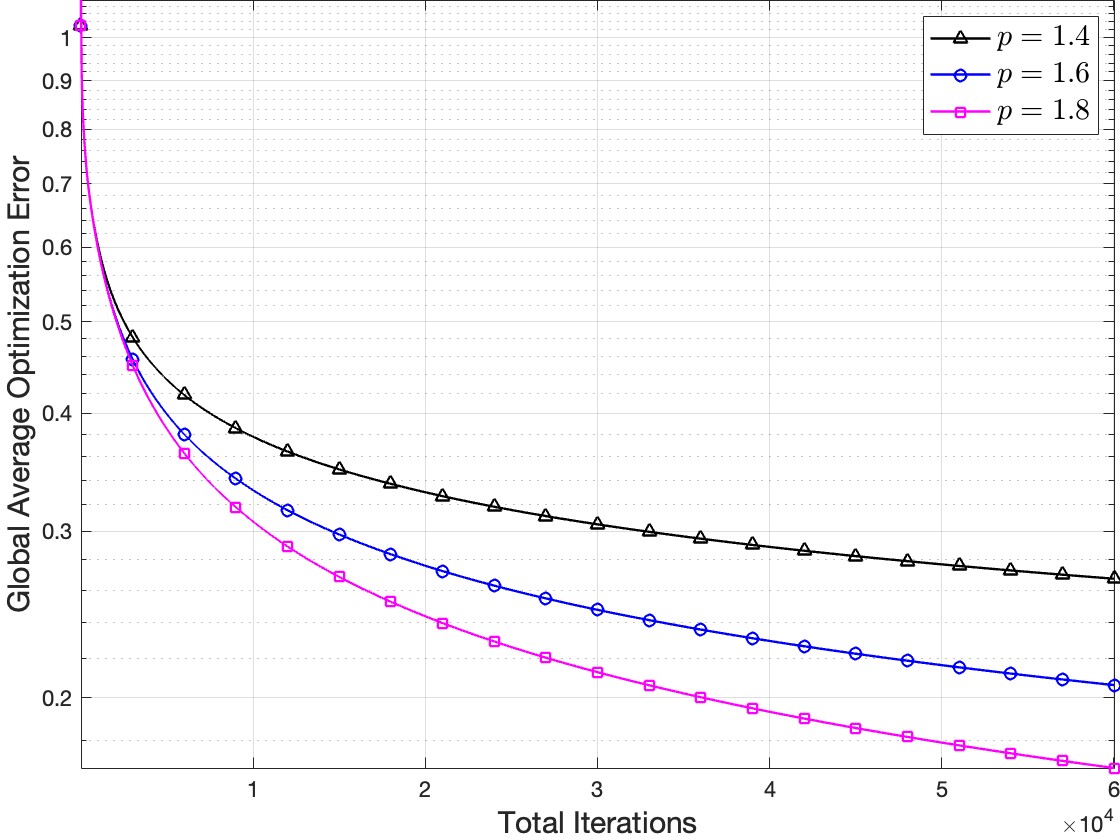}}
	\caption{Global average optimization error under different tail parameters $p$ of the noises.}
	\label{comp_p}
\end{figure}

\section{Conclusion}
	In this work, we have provided a clipped federated mirror descent algorithm for solving the DSO problem under heavy-tailed stochastic gradient noise. A theory of the algorithm is rigorously established. A high probability convergence rate of 
	$\huaO\left(T^{\f{1-p}{2p}}\log^\gamma T\right)$ ($1<p\leq2$, $\gamma>1$) is achieved. The study is also an extension of the existing work on federated optimization/learning under bounded-variance light-tailed noise. Several typical byproducts of the main algorithm such as Clipped FedSGD and Clipped FedSED corresponding to different selections of mirror map are clearly described. Throughout the analysis approaches, we do not require the compactness assumption of the decision domain adopted by many existing work on DMD. Numerical experiments are conducted to verify the convergence performance of the main algorithm. Due to space constraints, this paper only considers a federated optimization algorithm with one type of communication patterns. In the future, it would be interesting and possible to establish high probability convergence results for other types of federated optimization/learning algorithms under heavy-tailed-noise conditions of the stochastic gradient.

\end{CJK*}
\end{document}